\documentclass[11pt,reqno]{amsart}
 \usepackage{amssymb,amsfonts,amscd,amsbsy, color, amsmath}
\usepackage[mathscr]{eucal}
\usepackage{url}
\usepackage{graphicx}
\usepackage{mathtools}
\usepackage{todonotes}
\usepackage{tabularx}

\newtheorem{theorem}{Theorem}[section]

\newtheorem{problem}[theorem]{Problem}
\theoremstyle{definition}

\numberwithin{equation}{section}

\newcommand{\Z}{\mathbb{Z}}

\makeatletter
\def\imod#1{\allowbreak\mkern5mu({\operator@font mod}\,\,#1)}
\makeatother
\allowdisplaybreaks

\makeatletter
\@namedef{subjclassname@2020}{%
  \textup{2020} Mathematics Subject Classification}
\makeatother

\begin{document}

\title[A $q$-multisum identity arising from finite chain ring probabilities]{A $q$-multisum identity arising from finite chain ring probabilities}

\author{Jehanne Dousse}

\author{Robert Osburn}

\address{Univ Lyon, CNRS, Universit{\'e} Claude Bernard Lyon 1, UMR5208, Institut Camille Jordan, F-69622 Villeurbanne, France}

\address{School of Mathematics and Statistics, University College Dublin, Belfield, Dublin 4, Ireland}

\email{dousse@math.cnrs.fr}

\email{robert.osburn@ucd.ie}

\subjclass[2020]{16P10, 16P70, 33D15}
\keywords{Finite chain ring, Hecke-type series, Appell-Lerch sums, $q$-multisum identity}

\date{\today}

\begin{abstract}
In this note, we prove a general identity between a $q$-multisum $B_N(q)$ and a sum of $N^2$ products of quotients of theta functions. The $q$-multisum $B_N(q)$ recently arose in the computation of a probability involving modules over finite chain rings.
\end{abstract}

\maketitle

\section{Introduction}

Probabilistic proofs of classical $q$-series identities constitute an intriguing part of the literature in combinatorics. A prominent example of this perspective concerns the Andrews-Gordon identities \cite{and, gor} which state for $1 \leq i \leq k$ and $k \geq 2$

\begin{equation} \label{AG}
\sum_{n_1, \dotsc, n_{k-1} \geq 0} \frac{q^{N_1^2 + \cdots + N_{k-1}^2 + N_1 + \cdots + N_{k-1}}}{(q)_{n_1} \cdots (q)_{n_k}} = \prod_{\substack{s=1 \\ s \not \equiv 0, \pm i \pmod{2k+1}}}^{\infty} \frac{1}{1-q^s},
\end{equation}
where $N_j = n_j + \cdots + n_{k-1}$. Here and throughout, we use the standard $q$-hypergeometric (or ``$q$-Pochhammer symbol") notation

\begin{equation*}
(a)_n = (a;q)_n := \prod_{k=0}^{n-1} (1-aq^{k}),
\end{equation*}
valid for $n \in \mathbb{N} \cup \{\infty \}$. In \cite{ful}, Fulman uses a Markov chain on the nonnegative integers to prove the extreme cases $i=1$ and $i=k$ of (\ref{AG}). Chapman \cite{chap} cleverly extends Fulman's methods to prove (\ref{AG}) in full generality. In \cite{cohen}, Cohen explicitly computes probability laws of $p^{\ell}$-ranks of finite abelian groups to give a group-theoretic proof of (\ref{AG}). For a generalization of this computation, see \cite{del}. In this note, we are interested in a recent probability computation with a ring-theoretic flavor as it leads to an expression similar to the left-hand side of (\ref{AG}). 

Our focus is on finite chain rings, a notion we now briefly recall (for further details, see Section 2 in both \cite{BHKW} and \cite{hl}). A ring is called a left (resp. right) chain ring if its lattice of left (resp. right) ideals forms a chain. Any finite chain ring is a local ring, i.e., it has a unique maximal ideal which coincides with its radical. Let $\mathcal{R}$ be a finite chain ring with radical $\mathcal{N}$, $q$ be the order of the residue field $\mathcal{R}/\mathcal{N}$ and $N$ be the index of nilpotency of $\mathcal{N}$. Recently, the authors of \cite{BHKW} expressed the density $\psi(n,k,q,N)$ of free submodules $\mathcal{M}$ of $\mathcal{R}^n$ (over $\mathcal{R}$) of length $k:=\log_{q} (|\mathcal{M}|)$ as $n \to \infty$ as the reciprocal of the $q$-multisum (replacing $1/q$ in their notation with $q$)

\begin{equation}
\label{eq:BHKWsum}
B_N(q) :=\sum_{\substack{K_2, \dots , K_N \geq 0\\N |K_2+\dots+K_N}} \frac{q^{K_2^2 + \cdots + K_N^2 - (K_2 + \cdots +K_N)^2/N}}{(q)_{k_2} \cdots (q)_{k_N}},
\end{equation}
where $N \geq 2$ is an integer and $K_i = \sum_{j=2}^i k_j$. Upper and lower bounds for $B_N(q)$ are obtained and then used to show (under suitable conditions) that $\psi(n,k,q,N)$ is at least $1-\epsilon$ where $0 < \epsilon < 1$ (see Theorems 6 and 8, respectively, in \cite{BHKW}). Moreover, we have
 \begin{equation} \label{slater83}
B_2(q) = \prod_{\substack{s=1 \\ s \equiv \pm 2, \pm 3, \pm 4, \pm 5 \pmod{16}}}^{\infty} \frac{1}{1-q^s},
\end{equation}
which is (S.83) in \cite{McSZ}. In view of (\ref{AG}) and (\ref{slater83}), the authors in \cite{BHKW} posed the following (slightly rewritten) problem.
\begin{problem} \label{prob}
Determine whether $B_N(q)$ can be expressed as a product of $q$-Pochhammer symbols.
\end{problem}

The purpose of this note is to solve Problem \ref{prob}. It turns out that the solution is slightly more involved than either (\ref{AG}) or (\ref{slater83}), namely $B_N(q)$ is a sum of $N^2$ products of quotients of theta functions (but not a single product of $q$-Pochhammer symbols, for general $N$). Before stating our main result, we recall some further standard notation: 
\begin{align*}
j(x;q) &:=(x)_{\infty}(q/x)_{\infty}(q)_{\infty}, \\
j(x_1, x_2, \dotsc, x_n;q) &:= j(x_1;q) j(x_2;q) \cdots j(x_n;q), \\
J_{a,m} &:= j(q^a;q^m), \\
\overline{J}_{a,m} &:= j(-q^a;q^m),\\
J_{m} &:= J_{m,3m} = (q^m;q^m)_{\infty}.
\end{align*}
Note that these quantities are products of $q$-Pochhammer symbols. Our main result is now the following.

\begin{theorem}
\label{main}
For all $N \geq 2$, we have
\begin{align} \label{theta}
&B_N(q)= \frac{1}{(q)_{\infty}^2 \overline{J}_{0,N(N+2)}} \sum_{r=0}^{N-1} \sum_{s=0}^{N-1} \frac{(-1)^{r+s+1}q^{{r\choose 2} +{s+1\choose 2} + r(s+1)(N+1)+r+s+1} J_{N^2(N+2)}^3}{j((-1)^Nq^{N(N+2)r+N(N+3)/2};q^{N^2(N+2)})} \\ 
& \qquad \qquad \qquad \qquad \qquad \qquad \qquad \qquad \times \frac{j(-q^{N(s-r)};q^{N^2})j(q^{N(N+2)(r+s)+N(N+3)};q^{N^2(N+2)})}{j((-1)^Nq^{N(N+2)s+N(N+3)/2};q^{N^2(N+2)})}. \nonumber
\end{align}
\end{theorem}

Formula (\ref{theta}) is of interest for at least two reasons. First, Andrews-Gordon type $q$-multisums akin to (\ref{AG}) are typically evaluated as single infinite products using $q$-series methods such as Bailey pairs, the triple product identity or the quintuple product identity. Instances of $q$-multisums which evaluate to sums of infinite products seem to be less well-studied and thus certainly require further attention. For pertinent work involving character formulas of irreducible highest weight modules of Kac-Moody algebras of affine type, see \cite{dk1, dk3}. Second, in order to compute asymptotics or find congruences for the coefficients of $q$-multisums, one would ideally prefer a single infinite product expression. In lieu of this situation, sums of infinite products are often still helpful. Indeed, contrarily to \eqref{eq:BHKWsum} which requires computing a $(N-1)$-fold sum, (\ref{theta}) only involves a double sum. As a comparison with Table 1 in \cite{BHKW}, we explicitly compute $B_N(q)$ for $2 \leq N \leq 10$ and $N=100$ and $1/q=2,3,5,7,11$ to five decimals with Maple using \eqref{theta}. Table \ref{tab: moreN} above suggests that when $q \to 0$, the limiting value of $B_N(q)$ is $1$. This statement is confirmed in \cite[Corollary 10, (1)]{BHKW}.

The paper is organized as follows. In Section 2, we recall one of the main results from \cite{HM}, then prove Theorem \ref{main}. In Section 3, we make some concluding remarks.

\begin{table}
\begin{tabular}{|c|c|c|c|c|c|}
 \hline
$N \ \setminus \ 1/q$ & $2$ & $3$ & $5$ & $7$ & $11$ \\ 
 \hline
$2$ & $0.59546$ &$0.84191$ &$0.95049$ &$0.97627$ &$0.99092$   \\ 
\hline 
$3$ & $0.47084$ &$0.79666$ &$0.94102$ &$0.97295$ & $0.99010$  \\ 
\hline 
$4$ & $0.42109$ &$0.78230$ &$0.93915$ &$0.97248$ &$0.99002$   \\ 
\hline 
$5$ & $0.39877$ & $0.77759$&$0.93877$ &$0.97241$ &$0.99002$   \\ 
\hline 
$6$ & $0.38819$ &$0.77603$ &$0.93870$ &$0.97240$ &$0.99002$   \\ 
\hline 
$7$ & $0.38304$ &$0.77551$ &$0.93868$ &$0.97240$ &$0.99002$   \\ 
\hline 
$8$ & $0.38050$ &$0.77533$ &$0.93868$ &$0.97240$ &$0.99002$   \\ 
\hline
$9$ & $0.37924$ &$0.77528$ &$0.93868$ &$0.97240$ &$0.99002$   \\ 
\hline 
$10$ & $0.37861$ &$0.77526$ &$0.93868$ &$0.97240$ & $0.99002$  \\ 
\hline
\hline 
$100$ & $0.37798$ &$0.77525$ &$0.93868$ &$0.97240$ & $0.99002$  \\ 
\hline
\hline
$(q)_{\infty}$ & $0.28879$ &$0.56013$ &$0.76033$ &$0.83680$ & $0.90083$  \\ 
\hline
\end{tabular}
\caption{\label{tab: moreN} Values of $B_N(q)$}
\end{table}

\section{Proof of Theorem \ref{main}}

Before the proof of Theorem \ref{main}, we need to recall some background from the important work of Hickerson and Mortenson \cite{HM}. First, we employ the Hecke-type series

\begin{equation} \label{f}
f_{a,b,c}(x,y,q) := \Bigl( \sum_{r,s \geq 0} - \sum_{r,s < 0} \Bigr) (-1)^{r+s} x^r y^s q^{a\binom{r}{2} + brs + c \binom{s}{2}}.
\end{equation}
Next, consider the Appell-Lerch series

\begin{equation} \label{m}
m(x,q,z) := \frac{1}{j(z;q)} \sum_{r \in \mathbb{Z}} \frac{(-1)^r q^{\binom{r}{2}} z^r}{1-q^{r-1}xz},
\end{equation}
where $x$, $z \in \mathbb{C}^{*} := \mathbb{C} \setminus \{0\}$ with neither $z$ nor $xz$ an integral power of $q$ in order to avoid poles. One of the main results in \cite{HM} expresses (\ref{f}) in terms of (\ref{m}). Let

\begin{equation} \label{g}
\begin{aligned}
g_{a,b,c}(x,y,q,z_1,z_0) & := \sum_{t=0}^{a-1} (-y)^t q^{c\binom{t}{2}} j(q^{bt}x; q^a) m\Bigl(-q^{a\binom{b+1}{2} - c\binom{a+1}{2} - t(b^2 - ac)} \frac{(-y)^a}{(-x)^b}, q^{a(b^2 - ac)}, z_0 \Bigr) \\
& + \sum_{t=0}^{c-1} (-x)^t q^{a\binom{t}{2}} j(q^{bt}y; q^c) m\Bigl(-q^{c\binom{b+1}{2} - a\binom{c+1}{2} - t(b^2 - ac)} \frac{(-x)^c}{(-y)^b}, q^{c(b^2-ac)}, z_1 \Bigr).
\end{aligned}
\end{equation}
Following \cite{HM}, we use the term ``generic" to mean that the parameters do not cause poles in the Appell-Lerch sums or in the quotients of theta functions. 

\begin{theorem}[\cite{HM}, Theorem 1.3]  \label{hickmort} Let $n$ and $p$ be positive integers with $(n,p)=1$. For generic $x$, $y \in \mathbb{C}^{*}$,
$$f_{n,n+p,n}(x,y,q) = g_{n,n+p,n}(x,y,q,-1,-1) + \frac{1}{\overline{J}_{0,np(2n+p)}} \theta_{n,p}(x,y,q),$$
where
$$
\begin{aligned}
\theta_{n,p}(x,y,q) & := \sum_{r^{*}=0}^{p-1} \sum_{s^{*}=0}^{p-1} q^{n\binom{r-(n-1)/2}{2} + (n+p)(r - (n-1)/2)(s+ (n+1)/2) + n\binom{s+ (n+1)/2}{2}} (-x)^{r - (n-1)/2} \\
& \times \frac{(-y)^{s + (n+1)/2} J_{p^2(2n+p)}^3 j(-q^{np(s-r)} \frac{x^n}{y^n}; q^{np^2}) j(q^{p(2n+p)(r+s) + p(n+p)} (xy)^p; q^{p^2(2n+p)})}{j(q^{p(2n+p)r + p(n+p)/2} \frac{(-y)^{n+p}}{(-x)^n}, q^{p(2n+p)s + p(n+p)/2} \frac{(-x)^{n+p}}{(-y)^n}; q^{p^2(2n+p)})}.
\end{aligned}
$$
Here, $r:= r^{*} + \{(n-1)/2\}$ and $s:= s^{*} + \{ (n-1)/2 \}$ with $0 \leq \{ \alpha \} < 1$ denoting the fractional part of $\alpha$.
\end{theorem}

We can now prove Theorem \ref{main}.

\begin{proof}[Proof of Theorem \ref{main}]
The first step is to recognize $B_N(q)$ in a different context. For $N \geq 1$, consider the string function of level $N$ of the affine Lie algebra $A^{(1)}_1$ (e.g., see \cite{LP, SW})
\begin{equation}
\label{eq:string}
\mathcal{C}_{m, \ell}^N(q) = \frac{q^{\frac{m^2-\ell^2}{4N}}}{(q)_{\infty}} \sum_{\substack{\mathbf{n} \in \Z_{\geq 0}^{N-1}\\ \frac{m+\ell}{2N}+(C^{-1}\mathbf{n})_1 \in \Z}} \frac{q^{\mathbf{n}C^{-1}(\mathbf{n} - \mathbf{e}_{\ell})^{T}}}{(q)_{n_1} \cdots (q)_{n_{N-1}}},
\end{equation}
where $\mathbf{n}=(n_1, \dots, n_{N-1})$, $\mathbf{e}_i$ is the $i$-th standard unit vector in $\Z^{N-1}$ (with $\mathbf{e}_0 =\mathbf{e}_N=0$), $C$ is the $A_{N-1}$ Cartan matrix whose inverse $C^{-1}$ is given by
$$(C^{-1})_{i,j} = \min (i,j) - \frac{ij}{N},$$
and $(C^{-1}\mathbf{n})_1$ is the first entry in the vector $C^{-1}\mathbf{n}$. A straightforward computation (see the proof of Theorem 5 in \cite{BHKW}) yields

\begin{equation} \label{rewrite}
B_N(q) = \sum_{\substack{\mathbf{n} \in \Z_{\geq 0}^{N-1}\\ (C^{-1}\mathbf{n})_1 \in \Z}} \frac{q^{\mathbf{n}C^{-1}\mathbf{n}^{T}}}{(q)_{n_1} \cdots (q)_{n_{N-1}}}.
\end{equation}
Comparing \eqref{eq:string} when $\ell = 0$ and $m$ is divisible by $2N$ with (\ref{rewrite}), we have for all $N \geq 2$,

\begin{equation}
\label{eq:equality}
B_N(q)= q^{\frac{-m^2}{4N}}(q)_{\infty}\mathcal{C}_{m, 0}^N(q).
\end{equation}

Next, by Example 1.3 on page 386 of \cite{HM}, we have
$$\mathcal{C}_{m, 0}^N(q) = \frac{1}{(q)_{\infty}^3} f_{1,N+1,1}(q^{1+m/2}, q^{1-m/2},q).$$
Thus from \eqref{eq:equality}, we deduce that for all $N \geq 2$ and $m$ divisible by $2N$,
\begin{equation}
\label{eq:almostfinal}
B_N(q)=  \frac{q^{\frac{-m^2}{4N}}}{(q)_{\infty}^2} f_{1,N+1,1}(q^{1+m/2}, q^{1-m/2},q).
\end{equation}
By Theorem \ref{hickmort}, we have 

$$
\begin{aligned}
f_{1,N+1,1}(q^{1+m/2}, q^{1-m/2},q) & = g_{1,N+1,1}(q^{1+m/2}, q^{1-m/2},q,-1,-1) \\
& \qquad \qquad +  \frac{1}{\overline{J}_{0, N(N+2)}} \theta_{1,N}(q^{1+m/2}, q^{1-m/2},q).
\end{aligned}
$$
Now, observe that
$$g_{1,N+1,1}(q^{1+m/2}, q^{1-m/2},q,-1,-1)=0$$
as there are no poles in the Appell-Lerch series 

$$
m(q^{N(N+1)/2 + m(N+2)/2}, q^{N(N+2)}, -1)
$$
and
$$
m(q^{N(N+1)/2 -m(N+2)/2}, q^{N(N+2)}, -1)
$$
(indeed, this is true whenever $m(N+2)/2 \not\equiv \pm N(N+1)/2 \pmod{N(N+2)}$, which is always the case when $m \equiv 0 \pmod{2N}$) and 
$j(q^{1+m/2};q)=j(q^{1-m/2};q)=0$. Thus, 
$$B_N(q) = \frac{q^{\frac{-m^2}{4N}}}{(q)_{\infty}^2 \overline{J}_{0,N(N+2)}} \theta_{1,N}(q^{1+m/2}, q^{1-m/2},q).$$
We now take $m=0$. The factor $q^{\frac{-m^2}{4N}}$ disappears and $\theta_{1,N}(q, q,q)$ is given as in (\ref{theta}). This proves the result. 

\end{proof}

\section{Concluding remarks}

There are several avenues for further study. First, Table 1 suggests that as $N \to \infty$, the limiting value of $B_N(q)$ is strictly larger than $(q)_{\infty}$. This is a stronger statement than \cite[Corollary 10, (2)]{BHKW}. Thus, it would be desirable to compute both asymptotics for $B_N(q)$ and the correct limiting value of $\psi(n,k,q,N)$ as $N \to \infty$. Second, for $N=2$, $3$ and $4$, one can reduce the number of products of quotients of theta functions occurring in Theorem \ref{main} by first invoking Theorems 1.9--1.11 in \cite{HM}, then performing routine (yet possibly involved) simplifications \cite{FG}. In these cases, we require that $m \equiv 0 \pmod{2N}$, $m \not\equiv 0 \pmod{N(N+2)}$ and, if $m$ is odd, $m \not\equiv \pm (N+1) \pmod{2(N+2)}$. For example, one can recover (\ref{slater83}) in this manner. The details are left to the interested reader. Third, given that \eqref{eq:equality} is a key step in the proof of Theorem \ref{main}, it is natural to wonder if string functions which generalize \eqref{eq:string} (see \cite{hkkoty, kns}) can also be realized in terms of computing an appropriate probability. For recent related works on string functions, see \cite{mort2, mort1}. Finally, can Theorem \ref{main} be understood via Markov chains, group theory or, possibly, Hall-Littlewood functions \cite{stem}?

\section*{Acknowledgements}
The authors thank Ole Warnaar for pointing out references \cite{hkkoty, kns, LP}. The first author was partially supported by ANR COMBIN\'e ANR-19-CE48-0011 and the Impulsion grant of IdexLyon. The second author was partially supported by Enterprise Ireland CS20212030.

\end{document}